\documentclass[12pt]{article}
\usepackage[utf8]{inputenc}
\usepackage[a4paper, total={6in, 8in}, left=1.134in]{geometry}
\usepackage{amsmath}
\usepackage{amssymb}
\usepackage{amsthm}
\usepackage{hyperref}
\hypersetup{colorlinks=true,linkcolor=black,citecolor=black}

\newtheorem{theorem}{Theorem}[section]
\newtheorem{lemma}[theorem]{Lemma}
\newtheorem{proposition}[theorem]{Proposition}
\newtheorem{corollary}[theorem]{Corollary}

\newtheorem{remark}[theorem]{Remark}

\numberwithin{equation}{section}

\newcommand{\R}{\mathbb{R}}

\newcommand{\Q}{\mathbb{Q}}
\newcommand{\C}{\mathbb{C}}
\renewcommand{\P}{\mathbb{P}}

\renewcommand{\d}{\partial}
\newcommand{\dbar}{\overline{\partial}}

\DeclareMathOperator{\Ric}{Ric}
\DeclareMathOperator{\Rm}{Rm}
\DeclareMathOperator{\tr}{tr}
\DeclareMathOperator{\Kah}{Kah}
\DeclareMathOperator{\Kod}{Kod}

\DeclareMathOperator{\Vol}{Vol}
\DeclareMathOperator{\Real}{Re}

\newcommand{\oo}[1]{\overline{#1}}
\newcommand{\parabolic}[1]{\left(\frac{\d}{\d t} - \Delta_{#1}\right)}
\newcommand{\norm}[1]{\left\|#1\right\|}
\newcommand{\ip}[1]{\left\langle#1\right\rangle}
\newcommand{\comment}[1]{}

\title{Global Ricci Curvature Behaviour for the K\"ahler-Ricci Flow with Finite Time Singularities}
\author{Alexander Bednarek\\The University of Sydney\\\texttt{alexander.bednarek@sydney.edu.au}}
\date{\today}

\begin{document}

\maketitle

\begin{abstract}
\noindent We consider the K\"ahler-Ricci flow $(X, \omega(t))_{t \in [0,T)}$ on a compact manifold where the time of singularity, $T$, is finite. We assume the existence of a holomorphic map from the K\"ahler manifold $X$ to some analytic variety $Y$ which admits a K\"ahler metric on a neighbourhood of the image of $X$ and that the pullback of this metric yields the limiting cohomology class along the flow. This is satisfied, for instance, by the assumption that the initial cohomology class is rational, i.e., $[\omega_0] \in H^{1,1}(X,\Q)$. Under these assumptions we prove an $L^4$-like estimate on the behaviour of the Ricci curvature and that the Riemannian curvature is Type $I$ in the $L^2$-sense.
\end{abstract}

\section{Introduction}

Suppose $M$ is a compact real manifold. Recall that a Ricci flow $(M, g(t))_{t\in [0,T)}$ with finite time singularity $T < \infty$ is said to be of Type I, \cite{CK04}, if there exists a $C > 0$ such that
\begin{equation*}
    \sup_{M\times[0,T)} (T-t)|\Rm| \leq C.
\end{equation*}
This is conjectured to be the general setting for K\"ahler manifolds at least in the sense of scalar curvature, \cite[Conj 1.1]{JST23}, and we attempt to provide some general results in this direction.\\
\indent The Ricci flow was introduced by Hamilton in $1982$ in a program to study the geometric and topological structures of $3$-manifolds, which lead to a proof of Thurston's Geometrization conjecture and the subsequent Poincare conjecture by Perelman in $2003$, \cite{P02, P031, P032}. The K\"ahler-Ricci flow, i.e., the Ricci flow in the setting of K\"ahler manifolds, has been extensively studied since it's first appearance in Cao's alternate proof of the Calabi-Yau conjecture, \cite{C85}, a result which provided ground-breaking progress in the prescribed Ricci curvature problem, \cite{Y78, T00}. The K\"ahler-Ricci flow has since been viewed as a powerful tool to study the geometric structure of K\"ahler manifolds and this has been accentuated by the Analytic Minimal Model Program proposed over a decade ago by Song and Tian, \cite{ST07}, a geometric-analytic version of Mori's Minimal Model Program for the birational classification of algebraic varieties.\\
\indent The K\"ahler-Ricci flow is defined as follows. Suppose $(X, \omega_0)$ is a compact K\"ahler manifold. A family of K\"ahler metrics $\{\omega(t): t \in [0,T)\}$ on $X$ satisfies the (normalized) K\"ahler-Ricci flow if
\begin{equation}\label{KRF}\tag{KRF}
    \frac{\d \omega(t)}{\d t} = -\Ric(\omega(t)) - \omega(t), \quad \omega(0) = \omega_0.
\end{equation}
The normalization gives nice cohomological behaviour to the flow.\\ \indent Let us discuss some recent progress in understanding curvature behaviour along the K\"ahler-Ricci flow. For flows with finite time singularities, the pair of recent papers by Hallgren, Jian, Song and Tian, \cite{JST23, HJST23}, established a local Type I estimate near Ricci vertices for the scalar curvature along the K\"ahler-Ricci flow under the assumption that the initial K\"ahler class is rational, i.e., $[\omega_0] \in H^{1,1}(X,\Q)$. Additionally, Zhang, \cite{Z10}, proves that under the assumption of a map $f:X\to Y$ from the K\"ahler manifold $X$ to some analytic variety $Y$, if there exists a smooth K\"ahler metric $\omega_Y$ near the image of $X$ satisfying $[f^*\omega_Y] = [\omega(T)]$ then the scalar curvature is uniformly bounded above by $\frac{C}{(T-t)^2}$.\\
\indent On the other hand, when the flow exists for all time, under the assumption that the canonical line bundle $K_X$ is semi-ample, the behaviour of the scalar curvature has been extensively studied and is well understood, with some results about the Ricci curvature. Song and Tian have proven that the scalar curvature is uniformly bounded in this case, \cite{ST16}, and Jian demonstrated that the scalar curvature converges to the negative of the Kodaira dimension on the regular part of the induced Calabi-Yau fibration, \cite{J20}. Most recently, Hein, Lee and Tosatti have shown that for intermediate Kodaira dimension, i.e., $0 < \Kod(X) < \dim X$, the Ricci curvature is uniformly bounded away from the singular fibres of the fibration, \cite{HLT24}.\\
\indent There is also some global control of the Ricci curvature. Tian and Zhang, \cite{TZ15}, prove that for compact K\"ahler manifolds with semi-ample canonical bundle and positive Kodaira dimension there exists a constant $C > 0$ such that for any $\tau \in [0, \infty)$
\begin{equation*}
    \int_{\tau}^{\tau+1} \int_X |\Ric(\omega(t))|^4 \omega(t)^n dt \leq C.
\end{equation*}
We also note that for the Ricci flow of a $4$-dimensional, real, compact manifold $M$ with finite time singularity, Simon, \cite{S15}, has proven the following $L^4$-like bound on the Ricci curvature under the assumption that the scalar curvature is uniformly bounded
\begin{equation*}
    \int_0^T \int_M |\Ric(g(t))|^4 dV_{g(t)} dt \leq C.
\end{equation*}
However, in the K\"ahler-Ricci flow with finite time singularity, the scalar curvature cannot be uniformly bounded and must blow-up towards positive infinity, \cite{Z10}. Simon presents a more general theorem without this assumption on scalar curvature which one can use to derive equation (\ref{alternate_Ricci_bound}) in the complex dimension $2$ case.\\
\indent We consider the K\"ahler-Ricci flow with finite time singularities in a geometric setting which mirrors the semi-ample case and use Zhang's $C^0$-estimates in \cite{Z10} in combination with methods by Tian and Zhang from \cite{TZ15} and \cite{TZ16} to deduce the following result.

\begin{theorem}\label{main_theorem}
Suppose $(X, \omega_0)$ is a compact K\"ahler manifold and $(X,\omega(t))_{t\in[0,T)}$ is a solution to (\ref{KRF}) or (\ref{UKRF}) where $T < \infty$. Suppose there exists a map $f: X \to Y$ where $Y$ is an analytic variety and there exists a smooth K\"ahler metric $\omega_Y$ in a neighbourhood of the image $f(X)$ satisfying $[f^*\omega_Y] = [\omega(T)]$. Then there exists a constant $C = C(\omega_0,f^*\omega_Y,T) > 0$ such that
\begin{equation*}
    \int_0^T \int_X (T-t)^8|\Ric(\omega(t))|^4_{\omega(t)} \omega(t)^n dt \leq C,
\end{equation*}
and
\begin{equation*}
    \int_X |\Rm(\omega(t))|^2_{\omega(t)} \omega(t)^n  \leq \frac{C}{(T-t)^2},
\end{equation*}
and
\begin{equation*}
    \int_0^T \int_X (T-t)^8 |\nabla \Ric|_{\omega(t)}^2 \omega(t)^n dt \leq C,
\end{equation*}
and
\begin{equation*}
    \int_0^T \int_X (T-t)^6 |\nabla R|_{\omega(t)}^2 \omega(t)^n dt \leq C. 
\end{equation*}
\end{theorem}

\indent The contents of this paper are summarized as follows. In section $2$ an explanation of the geometric and cohomological assumptions is given together with a brief summary of global $C^0$-estimates which are essential in the computations which ensue. Next, in section $3$, we prove Theorem \ref{main_theorem} and provide many estimates on the Ricci potential. Section $4$ improves the $L^4$-like estimate on the behaviour of the Ricci curvature along the flow to a genuine $L^4$-estimate, (\ref{actual_L4_Ricci_bound}), which is dependent on the behaviour of the Laplacian of $\tr_{\omega(t)}(f^*\omega_Y)$. Finally, section $5$ provides a conversion of the $L^4$-like estimate on the Ricci curvature for the normalized K\"ahler-Ricci flow to one for the unnormalized flow and an appendix is included for the calculation of some $4$-th order commutation relations of covariant derivatives.\\

\noindent \textbf{Acknowledgements.} \textit{The author is extremely grateful to his supervisor Zhou Zhang for all the help in the writing of this paper. Particularly, for the reminder about the selected Ricci potential and its $C^0$-estimates. Additionally, the structure of this paper was greatly improved thanks to his suggestions. The author thanks Gang Tian for his interest and Zhenlei Zhang for discussion about \cite{TZ15, TZ16} and his kind words. Hosea Wondo also needs recognition for the many insightful discussions during the paper's conception.}

\section{The Geometric Setting and $C^0$-Behaviour}

Let $(X, \omega_0)$ be a compact K\"ahler manifold of dimension $\dim_\C(X) = n$. We consider the following normalized K\"ahler-Ricci flow
\begin{equation*}
    \frac{\d \omega(t)}{\d t} = -\Ric(\omega(t)) - \omega(t), \quad \omega(0) = \omega_0.
\end{equation*}
It is well known that the flow exists as long the the evolving cohomology class $[\omega(t)] = e^{-t}[\omega_0] + (1-e^{-t})[K_X]$ is positive, \cite{TZ06}, and the time of singularity, i.e. maximal existence time, is given by
\begin{equation*}
    T = \sup\{t\in (0,\infty]: e^{-t}[\omega_0] + (1-e^{-t})[K_X] \in \Kah(X)\},
\end{equation*}
where $\Kah(X)$ denotes the K\"ahler cone of $X$. We shall assume $T < \infty$ which corresponds with the case where the canonical line bundle $K_X$ is not nef.\\
Let us also assume that there exists a holomorphic map $f: X \to Y$, where $Y$ is an analytic variety that is smooth in a neighbourhood of the image $f(X)$, and that there exists a K\"ahler metric $\omega_Y$ on this neighbourhood such that $[f^*\omega_Y] = [\omega(T)]$. Note that although $\omega(T)$ is undefined, the limiting class $[\omega(T)]$ is still defined by $[\omega(T)] := e^{-T}[\omega_0] + (1-e^{-T})[K_X]$.\\
Under these assumptions we can choose a closed, smooth $(1,1)$-form $\chi$ representing $[K_X]$ such that 
\begin{equation*}
    f^* \omega_Y = e^{-T} \omega_0 + (1-e^{-T}) \chi
\end{equation*}
and a smooth volume form $\Omega$ such that $\chi = -\Ric(\Omega) = i\d\dbar \log \Omega$. By a standard reduction argument involving the $\d\dbar$-lemma, e.g. see \cite{Y78, TZ06}, the K\"ahler-Ricci flow (\ref{KRF}) is equivalent to the following complex Monge-Ampere type equation
\begin{equation}\label{nSKRF}\tag{SKRF}
    \frac{\d \varphi}{\d t} = \log \frac{(e^{-t}\omega_0 + (1-e^{-t})\chi + i \d\dbar \varphi)^n}{\Omega} - \varphi, \quad \varphi(0) = 0
\end{equation}
where $\omega(t) = e^{-t}\omega_0 + (1-e^{-t})\chi + i\d\dbar \varphi(t)$.

\begin{remark}
This situation is typically achieved by the assumption that $[\omega_0]$ is rational, i.e. $[\omega_0] \in H^{1,1}(X,\Q)$, so that by Kawamata's rationality and base-point free theorems we know $[\omega(T)] \in H^{1,1}(X,\Q)$ and we have an induced holomorphic map $f: X \to Y \subseteq \C\P^N$ satisfying $[f^* \omega_Y] = [\omega(T)]$, where $\omega_Y$ is now a multiple of the restriction of the Fubini-Study metric on $\C\P^N$ to $Y$, \cite{Z10}.
\end{remark}

\noindent Also, it is well-known that there exists a constant $C > 0$ such that $\omega(t)^n \leq C \Omega$. This means the volume is bounded along the flow, i.e.,
\begin{equation*}
    \Vol_{\omega(t)}(X) = \int_X \omega(t)^n \leq C.
\end{equation*}
\noindent A quick calculation, \cite{Z10}, shows that
\begin{equation}\label{Ricci_formula}
    (1-e^{t-T})\Ric(\omega(t)) = e^{t-T} \omega(t) - f^* \omega_Y -i\d\dbar u(t),
\end{equation}
where $u = (1-e^{t-T})\frac{\d \varphi}{\d t} + \varphi$, and for the scalar curvature,
\begin{equation}\label{scalar_formula}
    (1-e^{t-T}) R = ne^{t-T} - \tr_{\omega(t)} (f^*\omega_Y) - \Delta u.
\end{equation}
For this reason, i.e. equation (\ref{Ricci_formula}), we shall call $u$ the \textbf{Ricci potential}. One can observe that letting $T \to \infty$ yields the same Ricci potential and decompositions, (\ref{Ricci_formula}), (\ref{scalar_formula}), of the Ricci and scalar curvatures in terms of the Ricci potential as in the semi-ample case. Additionally, one can calculate the following evolution equations, \cite{Z10},
\begin{align*}
    &\parabolic{} u = -n + \tr_{\omega(t)}(f^* \omega_Y),\\
    &\parabolic{} |\nabla u|^2 = |\nabla u|^2 - |\nabla \nabla u|^2 - |\nabla \oo{\nabla} u|^2 + 2 \Real\ip{\nabla \tr_{\omega(t)}(f^*\omega_Y), \oo{\nabla} u},\\
    &\parabolic{} \Delta u = \Delta u + \ip{\Ric, i\d\dbar u} + \Delta \tr_{\omega(t)}(f^*\omega_Y).
\end{align*}
This follows the setup in \cite{Z10} and the following theorem summarizes Zhang's results.

\begin{theorem}\label{bounds}
There exists a constant $C=C(\omega_0, f^*\omega_Y, T) > 0$ such that on $X \times [0,T)$ the following estimates hold
\begin{align*}
    |u| \leq C, \quad |\nabla u| \leq C, \quad \Delta u \leq C, \quad -\frac{C}{T - t} \leq \Delta u,\\
    \quad 0 \leq \tr_{\omega(t)} (f^*\omega_Y) \leq C, \quad -C \leq R, \quad R \leq \frac{C}{(T-t)^2},
\end{align*}
where all norms, gradients, Laplacians and scalar curvatures are taken with respect to the metric $\omega(t)$.
\end{theorem}

\section{Global $L^4$-Behaviour of the Ricci Curvature}

Let $\nabla$ and $\oo{\nabla}$ denote the $(1,0)$ and $(0,1)$-parts of the Chern connection of $(X,\omega(t))$, respectively. We shall use the convention that $C$ represents a positive constant which may change line to line and may represent different constants within the same line. We note that the constants below will always depend on $\omega_0$, $f^*\omega_Y$ and $T$ due to Theorem \ref{bounds} unless otherwise stated. Additionally, we may drop the volume form $\omega(t)^n$ when manipulating integrals.

\subsection{Preliminary $L^2$-Estimates}

We begin with the following Proposition which is included in Theorem \ref{main_theorem}.

\begin{proposition}
There exists a constant $C > 0$ such that for all $t \in [0,T)$ the following estimate holds
\begin{equation}\label{Rm_bound}
    \int_X |\Rm|^2 \omega^n \leq \frac{C}{(T-t)^2},
\end{equation}
and additionally,
\begin{equation}\label{L1_scalar_bound}
    \int_X |R| \omega^n \leq C.
\end{equation}
\end{proposition}

\begin{proof}
Let us assume $n > 1$, although the statement still holds in the case $n = 1$. One can use Chern-Weil theory (see Kobayashi \cite[\S 4.4]{K14}), to express Chern forms, $c_1, c_2$, dependent on $t$, in terms of the Riemannian, Ricci and scalar curvatures as
\begin{align*}
    c^2_1 \wedge \omega^{n-2} &= \frac{1}{4\pi^2 n(n-1)} (R^2 - |\Ric|^2) \omega^n,\\
    c_2 \wedge \omega^{n-2} &= \frac{1}{8\pi^2 n(n-1)} (R^2 - 2|\Ric|^2 + |\Rm|^2) \omega^n.
\end{align*}
Thus
\begin{equation*}
    \int_X (c_2-c_1^2) \wedge \omega^{n-2} = \frac{1}{8\pi n(n-1)} \int_X (|\Rm|^2-R^2)\omega^n,
\end{equation*}
so that the left-hand side integrand is purely cohomological and thus the integral is a topological invariant for all $t\in [0,T]$ as the class $[\omega(T)]$ is still well-defined. Furthermore, the integral on the left-hand side is a continuous function from $[0,T]$ to $\R$, so it is bounded above and below. Therefore
\begin{equation*}
    \int_X |\Rm|^2 \leq C+\int_X R^2.
\end{equation*}
Next, note that
\begin{equation*}
    \int_X R \omega^n = \int_X n \Ric \wedge \omega^{n-1} = 2 \pi n\int_X c_1 \wedge \omega^{n-1}
\end{equation*}
so that by the same argument
\begin{equation*}
    \int_X R \omega^n \leq C.
\end{equation*}
Additionally, as $R$ is uniformly bounded below, we have
\begin{equation*}
    \int_X |R| \omega^n \leq C.
\end{equation*}
By Theorem \ref{bounds} one finds
\begin{align*}
    \int_X R^2 \leq \frac{C}{(T-t)^2} \int_X |R|\leq \frac{C}{(T-t)^2}
\end{align*}
and the result follows.
\end{proof}

\noindent 
The above argument is essentially entirely topological apart from the use of Theorem \ref{bounds} in estimating the $L^2$-norm of the scalar curvature. Additionally, the $L^2$-behaviour of the Riemann curvature tensor is entirely equivalent to the $L^2$-behaviour of the scalar curvature as a result. This confirms that the Riemannian curvature is indeed Type I in the $L^2$-sense.\\

\noindent The next two results provide essential control on the complex Hessians of the Ricci potentials to be used throughout.

\begin{proposition}\label{L2_mixed_hessian_bound}
There exists a constant $C > 0$ such that for all $t \in [0,T)$ we have
\begin{equation*}
    \int_X |\nabla \oo{\nabla} u|^2 \omega^n \leq \frac{C}{(T-t)^2}.
\end{equation*}
\end{proposition}

\begin{proof}
Using Theorem \ref{bounds}, one can calculate in normal coordinates
\begin{align*}
    \int_X |\nabla \oo{\nabla} u|^2 &= \int_X \nabla_i \nabla_{\oo{j}} u \cdot \nabla_{\oo{i}} \nabla_{j}  u\\
    &= \int_X -\nabla_{\oo{j}} u \cdot \nabla_i \nabla_{\oo{i}} \nabla_{j} u\\
    &= \int_X - \nabla_{\oo{j}} u \cdot \nabla_j \Delta u\\
    &= \int_X (\Delta u)^2\\
    &\leq \frac{C}{(T-t)^2}.
\end{align*}
\end{proof}

\noindent Take note that it is true for any arbitrary, smooth, real function $v : X \to \R$ that
\begin{equation*}
    \int_X |\nabla \oo{\nabla} v|^2 \omega^n = \int_X (\Delta v)^2 \omega^n.
\end{equation*}

\begin{proposition}\label{L2_unmixed_hessian_bound}
There exists a constant $C > 0$ such that for all $t \in [0,T)$ we have
\begin{equation*}
    \int_X |\nabla \nabla u|^2 \omega^n \leq \frac{C}{(T-t)^2}.
\end{equation*}
\end{proposition}

{\allowdisplaybreaks
\begin{proof}
Using Theorem \ref{bounds}, one can calculate in normal coordinates
\begin{align*}
    \int_X |\nabla \nabla u|^2 &= \int_X \nabla_i \nabla_j u \cdot \nabla_{\oo{i}} \nabla_{\oo{j}} u\\
    &= -\int_X \nabla_{\oo{i}} \nabla_i \nabla_j u \cdot \nabla_{\oo{j}} u\\
    &= -\int_X (\nabla_j \Delta u +  R_{j\oo{k}}\nabla_k u) \nabla_{\oo{j}} u\\
    &= \int_X (\Delta u)^2 - R_{j\oo{k}} \nabla_k u \nabla_{\oo{j}} u\\
    &\leq \int_X (\Delta u)^2 + |\Ric| |\nabla u|^2\\
    &\leq \frac{C}{(T-t)^2},
\end{align*}
where the last line uses the fact that $\norm{\Ric}_{L^1} \leq C\norm{\Ric}_{L^2} \leq C' \norm{\Rm}_{L^2}$ in combination with (\ref{Rm_bound}).
\end{proof}
}

\noindent We also require better control on the behaviour of the trace $\tr_\omega(f^*\omega_Y)$.

\begin{proposition}\label{gradient_trace_bound}
There exists a constant $C > 0$ such that for all $t \in [0,T)$ the following estimate holds
\begin{equation*}
    \int_X |\nabla \tr_\omega(f^* \omega_Y)|^2 \omega^n \leq \frac{C}{T-t}.
\end{equation*}
\end{proposition}

\begin{proof}
Note the following inequality, \cite[Eq. 3.19]{TZ15}, \cite{ST07}, which still holds in our case,
\begin{equation*}
    \Delta \tr_\omega (f^*\omega_Y) \geq -|\Ric| \tr_\omega (f^* \omega_Y) - C (\tr_\omega (f^* \omega_Y))^2 + \frac{|\nabla \tr_\omega (f^*\omega_Y)|^2}{\tr_\omega (f^*\omega_Y)}.
\end{equation*}
It follows that
\begin{align}\label{lower_trace_bound}
    |\nabla \tr_\omega(f^*\omega_Y)|^2 \leq C(\Delta \tr_\omega (f^*\omega_Y) + |\Ric|+1).
\end{align}
Therefore we have
\begin{equation*}
    \int_X |\nabla \tr_\omega (f^*\omega_Y)|^2 \leq C\int_X (|\Ric|+1) \leq \frac{C}{T-t},
\end{equation*}
since $\norm{\Ric}_{L^1} \leq C\norm{\Ric}_{L^2} \leq C' \norm{\Rm}_{L^2}$.
\end{proof}

\noindent The following estimate introduces an additional time integral which is then forced to be carried through the remaining estimates, in particular this includes the $L^4$-type estimate of the Ricci curvature. We provide some further analysis relevant to $\nabla \Delta u$ in the next section in an attempt to circumvent this issue.

\begin{proposition}
There exists a constant $C > 0$ such that
\begin{equation*}
    \lim_{\tau \to T} \int_0^\tau \int_X (T-t)^4 |\nabla \Delta u(t)|^2 \omega(t)^n dt \leq C.
\end{equation*}
\end{proposition}

\begin{proof}
Recall the evolution equation for $\Delta u$,
\begin{equation*}
    \parabolic{} \Delta u = \Delta u + \ip{\Ric, i\d\dbar u} + \Delta \tr_\omega (f^* \omega_Y)
\end{equation*}
from which one can calculate
\begin{equation*}
    \parabolic{} (\Delta u)^2 = 2 (\Delta u)^2 + 2 \Delta u \ip{\Ric, i\d\dbar u} + 2 \Delta u \Delta \tr_\omega( f^*\omega_Y) - 2 |\nabla \Delta u|^2.
\end{equation*}
Therefore,
\begin{align*}
    \int_X 2|\nabla \Delta u|^2 &= \int_X 2(\Delta u)^2 + 2\Delta u \ip{\Ric, i\d\dbar u} + 2\Delta u \Delta \tr_\omega (f^* \omega_Y) - \frac{\d}{\d t} (\Delta u)^2\\
    &\leq \int_X 2 (\Delta u)^2+2|\Delta u| |\Ric||\nabla \oo{\nabla} u| - 2\nabla_i \Delta u \cdot \nabla_{\oo{i}} \tr_\omega (f^* \omega_Y) - \frac{\d}{\d t} (\Delta u)^2 \\
    &\leq \int_X 2(\Delta u)^2+(\Delta u)^2|\nabla \oo{\nabla} u|^2 + |\Ric|^2 + |\nabla \Delta u|^2 + |\nabla \tr_\omega (f^*\omega_Y)|^2 - \frac{\d}{\d t} (\Delta u)^2.
\end{align*}
and so
\begin{align*}
    \int_X |\nabla \Delta u|^2 \leq \frac{C}{(T-t)^4} - \int_X \frac{\d}{\d t} (\Delta u)^2.
\end{align*}
Note that
\begin{align*}
    \frac{d}{dt} \int_X (\Delta u)^2 \omega^n &= \int_X \frac{\d}{\d t} (\Delta u)^2 \omega^n + \int_X (\Delta u)^2 (-R-n) \omega^n\\
    &\leq \int_X \frac{\d}{\d t} (\Delta u)^2 \omega^n + \frac{C}{(T-t)^2}
\end{align*}
so that one has
\begin{align*}
    \int_X |\nabla \Delta u|^2 \leq \frac{C}{(T-t)^4} - \frac{d}{dt} \int_X (\Delta u)^2
\end{align*}
and hence
\begin{align*}
    \int_X (T-t)^4 |\nabla \Delta u|^2 \leq C-\frac{d}{d t} \int_X (T-t)^4 (\Delta u)^2. 
\end{align*}
Integrating from $0$ to $\tau\in [0,T)$ gives
\begin{align*}
    \int_0^\tau \int_X (T-t)^4 |\nabla \Delta u(t)|^2 \omega(t)^n dt &\leq C\tau - \int_X (T-\tau)^4 (\Delta u(\tau))^2 \omega(\tau)^n + \int_X T^4 (\Delta u(0))^2 \omega_0^n \\
    &\leq CT + C\\
    &\leq C.
\end{align*}
The result follows as $C$ is independent of $\tau$.
\end{proof}

\noindent Note that one does not have to pull the factor of $(T-t)^4$ inside the integral but doing so yields a more pleasant $L^4$-like estimate on the Ricci curvature and estimates for some other quantities. Otherwise, we find there exists a $C > 0$ such that for all $\tau \in [0,T)$ we have
\begin{equation}\label{alternate_bound}
    \int_0^\tau \int_X |\nabla \Delta u|^2 \omega^n dt \leq \frac{C}{(T-\tau)^3}.
\end{equation}
The following result partially completes Theorem \ref{main_theorem}.

\begin{corollary}
There exists a constant $C > 0$ such that
\begin{equation*}
    \lim_{\tau\to T} \int_0^\tau \int_X (T-t)^6 |\nabla R|^2 \omega(t)^n dt \leq C.
\end{equation*}
\end{corollary}

\begin{proof}
Recall equation (\ref{scalar_formula}). From this we can see
\begin{equation*}
    (1-e^{t-T}) \nabla R = -\nabla \tr_\omega(f^*\omega_Y) - \nabla \Delta u
\end{equation*}
and thus
\begin{equation*}
    |\nabla R|^2 \leq \frac{C}{(T-t)^2}(|\nabla \tr_\omega(f^*\omega_Y)|^2 + |\nabla \Delta u|^2).
\end{equation*}
It follows that 
\begin{equation*}
    \int_X |\nabla R|^2 \leq \frac{C}{(T-t)^3} + \frac{C}{(T-t)^2} \int_X |\nabla \Delta u|^2.
\end{equation*}
Therefore there exists a constant $C > 0$ such that for any $\tau \in [0,T)$ one has
\begin{equation*}
    \int_0^\tau \int_X (T-t)^6|\nabla R|^2 \leq C
\end{equation*}
and the result follows.
\end{proof}

\subsection{The $L^4$-Behaviour of the Ricci Curvature}

\noindent Take note of the following general estimates, which appear in \cite[Lem. 4.3]{TZ16}, \cite[Lem. 3.5]{TZ15}, but are now appropriately adapted to our situation.

\begin{proposition}
Let $v: X \to \R$ be an arbitrary, smooth, real function. Then the following integral estimates hold
\begin{align*}
     \int_X (|\nabla \nabla v|^2 &+ |\nabla \oo{\nabla} v|^2)|\nabla \oo{\nabla} v|^2 \omega^n\\ 
     &\leq 24 \int_X (|\nabla v|^2 |\nabla \nabla \oo{\nabla} v|^2 + (\Delta v)^2 |\nabla \oo{\nabla} v|^2 +|\nabla v|^4 |\Ric|^2) \omega^n,\\
    \int_X (|\nabla \nabla v|^2 &+ |\nabla \oo{\nabla} v|^2)|\nabla \nabla v|^2 \omega^n\\
    &\leq 16\int_X (|\nabla v|^2 (|\nabla \nabla \nabla v|^2 + |\oo{\nabla} \nabla \nabla v|^2) + (\Delta v)^2 |\nabla \nabla v|^2 + |\nabla v|^4|\Ric|^2)\omega^n.
\end{align*}
\end{proposition}

\begin{proof}
By the Bochner formula, (\ref{Bochner_formula}),
\begin{align*}
    \int_X (|\nabla \nabla v|^2 &+ |\nabla \oo{\nabla} v|^2) |\nabla \oo{\nabla} v|^2\\
    &= \int_X (\Delta |\nabla v|^2 - \ip{\nabla \Delta v, \oo{\nabla} v} - \ip{\nabla v, \oo{\nabla} \Delta v} - \ip{\Ric, \oo{\nabla} v \otimes \nabla v}) |\nabla \oo{\nabla} v|^2.
\end{align*}
Using integration by parts for the first term gives
\begin{align*}
    \int_X \Delta |\nabla v|^2 \cdot |\nabla \oo{\nabla} v|^2 &= -\int_X \nabla_{\oo{i}} |\nabla v|^2 \cdot \nabla_i |\nabla \oo{\nabla} v|^2\\
    &=  - \int_X (\nabla_{\oo{i}} \nabla_j v \cdot \nabla_{\oo{j}} v + \nabla_j v \cdot \nabla_{\oo{i}} \nabla_{\oo{j}} v)\\
    &\quad\quad\quad\quad(\nabla_i \nabla_k \nabla_{\oo{l}} v \cdot \nabla_{\oo{k}} \nabla_l v + \nabla_k \nabla_{\oo{l}}v \cdot \nabla_i \nabla_{\oo{k}} \nabla_l v)\\
    &\leq \int_X (|\nabla \oo{\nabla} v| |\nabla v| + |\nabla v| |\nabla \nabla v|)(|\nabla \nabla \oo{\nabla} v| |\nabla \oo{\nabla} v| + |\nabla \oo{\nabla} v| |\nabla \nabla \oo{\nabla} v|)\\
    &\leq \int_X \frac{1}{2}(|\nabla \nabla v|^2 + |\nabla \oo{\nabla} v|^2) |\nabla \oo{\nabla} v|^2 + 4 |\nabla v|^2 |\nabla \nabla \oo{\nabla} v|^2.
\end{align*}
Integration by parts for the second term yields
{\allowdisplaybreaks
\begin{align*}
    -\int_X \ip{\nabla \Delta v, \oo{\nabla} v} |\nabla \oo{\nabla} v|^2 &= \int_X \Delta v (\nabla_k \nabla_{\oo{k}} v \cdot |\nabla \oo{\nabla} v|^2 + \nabla_{\oo{k}} v \cdot \nabla_k |\nabla \oo{\nabla} v|^2)\\
    &= \int_X (\Delta v)^2 |\nabla \oo{\nabla} v|^2 + \Delta v \cdot \nabla_{\oo{k}} v \cdot\\
    &\quad\quad\quad(\nabla_k \nabla_i \nabla_{\oo{j}} v \cdot \nabla_{\oo{i}} \nabla_j v + \nabla_i \nabla_{\oo{j}} v \cdot \nabla_k \nabla_{\oo{i}} \nabla_j v)\\
    &\leq \int_X (\Delta v)^2 |\nabla\oo{\nabla} v|^2 + |\Delta v| |\nabla v| (|\nabla \nabla \oo{\nabla} v||\nabla \oo{\nabla} v| + |\nabla \oo{\nabla} v| |\nabla \nabla \oo{\nabla} v|)\\
    &\leq \int_X 2(\Delta v)^2 |\nabla \oo{\nabla} v|^2 + |\nabla v|^2 |\nabla\nabla \oo{\nabla} v|^2. 
\end{align*}
}
Similarly,
\begin{align*}
    -\int_X \ip{\nabla v, \oo{\nabla} \Delta v} |\nabla \oo{\nabla} v|^2 \leq \int_X 2(\Delta v)^2 |\nabla \oo{\nabla} v|^2 + |\nabla v|^2 |\nabla\nabla \oo{\nabla} v|^2.
\end{align*}
The final term can be estimated by
\begin{align*}
    -\int_X \ip{\Ric, \oo{\nabla} v \otimes \nabla v} |\nabla \oo{\nabla} v|^2 &\leq \int_X |\Ric| |\nabla v|^2 |\nabla \oo{\nabla} v|^2\\
    &\leq \int_X |\Ric|^2|\nabla v|^4 + \frac{1}{4} |\nabla \oo{\nabla} v|^4.
\end{align*}
Combining these estimates yields the first inequality.\\
For the second we proceed in the same way. The Bochner formula, (\ref{Bochner_formula}), gives
\begin{align*}
    \int_X (|\nabla \nabla v|^2 &+ |\nabla \oo{\nabla} v|^2) |\nabla \nabla v|^2\\
    &= \int_X (\Delta |\nabla v|^2 - \ip{\nabla \Delta v, \oo{\nabla} v} - \ip{\nabla v, \oo{\nabla} \Delta v} - \ip{\Ric, \oo{\nabla} v\otimes \nabla v}) |\nabla \nabla v|^2.
\end{align*}
By integrating the first term by parts,
\begin{align*}
    \int_X \Delta |\nabla v|^2 \cdot |\nabla \nabla v|^2 &= -\int_X \nabla_{\oo{i}} |\nabla v|^2 \cdot \nabla_i|\nabla \nabla v|^2\\
    &= - \int_X (\nabla_{\oo{i}} \nabla_j v \cdot \nabla_{\oo{j}} v + \nabla_j v \cdot \nabla_{\oo{i}} \nabla_{\oo{j}} v)\\
    &\quad\quad\quad\quad (\nabla_i \nabla_k \nabla_l v\cdot \nabla_{\oo{k}}\nabla_{\oo{l}} v + \nabla_k \nabla_l v \cdot \nabla_i \nabla_{\oo{k}} \nabla_{\oo{l}} v)\\
    &\leq \int_X (|\nabla \oo{\nabla} v||\nabla v| + |\nabla v||\nabla \nabla v|)(|\nabla \nabla \nabla v| |\nabla\nabla v| + |\nabla \nabla v| |\oo{\nabla} \nabla \nabla v|)\\
    &\leq \int_X \frac{1}{2} (|\nabla \nabla v|^2 + |\nabla \oo{\nabla} v|^2)|\nabla \nabla v|^2 + 2|\nabla v|^2 (|\nabla\nabla \nabla v|^2 + |\oo{\nabla} \nabla \nabla v|^2).
\end{align*}
For the second term we find
{\allowdisplaybreaks
\begin{align*}
    -\int_X \ip{\nabla \Delta v, \oo{\nabla} v} |\nabla \nabla v|^2 &= \int_X \Delta v (\nabla_k \nabla_{\oo{k}} v \cdot |\nabla \nabla v|^2 + \nabla_{\oo{k}} v \cdot \nabla_k |\nabla \nabla v|^2)\\
    &= \int_X (\Delta v)^2 |\nabla \nabla v|^2 + \Delta v \cdot \nabla_{\oo{k}} v \cdot\\
    &\quad\quad\quad (\nabla_k \nabla_i \nabla_j v \cdot \nabla_{\oo{i}} \nabla_{\oo{j}} v + \nabla_i \nabla_j v \cdot \nabla_k \nabla_{\oo{i}} \nabla_{\oo{j}} v)\\
    &\leq \int_X (\Delta v)^2 |\nabla \nabla v|^2 + |\Delta v| |\nabla v| (|\nabla \nabla \nabla v| |\nabla \nabla v|+|\nabla \nabla v||\oo{\nabla}\nabla \nabla v|)\\
    &\leq \int_X 2(\Delta v)^2 |\nabla \nabla v|^2 + \frac{1}{2}|\nabla v|^2 (|\nabla \nabla \nabla v|^2 + |\oo{\nabla} \nabla \nabla v|^2).
\end{align*}
}
Similarly,
\begin{equation*}
    -\int_X \ip{\nabla v, \oo{\nabla}\Delta v} |\nabla \nabla v|^2 \leq \int_X 2(\Delta v)^2 |\nabla \nabla v|^2 + \frac{1}{2}|\nabla v|^2 (|\nabla \nabla \nabla v|^2 + |\oo{\nabla} \nabla \nabla v|^2).
\end{equation*}
Finally,
\begin{align*}
    -\int_X \ip{\Ric, \oo{\nabla} v \otimes \nabla v} |\nabla \nabla v|^2 &\leq \int_X |\Ric| |\nabla v|^2 |\nabla \nabla v|^2\\
    &\leq \int_X |\Ric|^2 |\nabla v|^4 + \frac{1}{4}|\nabla \nabla v|^4.
\end{align*}
The second inequality follows.
\end{proof}

\noindent Additionally, the following general estimate appears in \cite[Lem. 4.4]{TZ16}, \cite[Lem. 3.5]{TZ15}, and has been adapted to our situation.

\begin{proposition}\label{triple_nabla_bounds}
Let $v: X \times [0,T) \to \R$ be an arbitrary, smooth, real function and assume there exists there exists a constant $C' >0 $ such that $|\nabla v| \leq C'$ on $X \times [0,T)$. Then there exists a constant $C(n, C') > 0$ such that the following integral estimates hold
\begin{align*}
    \int_X |\nabla \nabla \oo{\nabla} v|^2 \omega^n &\leq C\int_X (|\nabla \Delta v|^2 + |\Rm|^2 + (\Delta v)^2 |\nabla \oo{\nabla} v|^2 )\omega^n,\\
    \int_X (|\nabla \nabla \nabla v|^2 + |\oo{\nabla} \nabla \nabla v|^2) \omega^n &\leq C \int_X (|\nabla \Delta v|^2 + |\Rm|^2 + (1+(\Delta v)^2)|\nabla \nabla v|^2 + |\nabla \Ric|^2)\omega^n.
\end{align*}
\end{proposition}

\begin{proof}
Observe that
{\allowdisplaybreaks
\begin{align*}
    \int_X |\nabla \nabla \oo{\nabla} v|^2 &= \int_X \nabla_i \nabla_j \nabla_{\oo{k}} v \cdot \nabla_{\oo{i}} \nabla_{\oo{j}} \nabla_{k} v\\
    &=- \int_X \nabla_j \nabla_{\oo{k}} v \cdot \nabla_i \nabla_{\oo{i}} \nabla_{\oo{j}} \nabla_{k} v \\
    &=-\int_X \nabla_j \nabla_{\oo{k}} v (\nabla_{\oo{j}} \nabla_k \Delta v + R_{i \oo{j}} \nabla_{\oo{i}} \nabla_k v + R_{k\oo{j}l\oo{i}} \nabla_i \nabla_{\oo{l}} v)\\
    &= \int_X \nabla_{\oo{j}}\nabla_j \nabla_{\oo{k}} v \cdot \nabla_k \Delta v - R_{i\oo{j}} \nabla_j\nabla_{\oo{k}} v\cdot \nabla_{\oo{i}} \nabla_k v - R_{k\oo{j}l\oo{i}}\nabla_j \nabla_{\oo{k}} v \cdot \nabla_i \nabla_{\oo{l}} v\\
    &\leq \int_X |\nabla \Delta v|^2 + C|\Rm| |\nabla \oo{\nabla} v|^2\\
    &\leq \int_X |\nabla \Delta v|^2 + C|\Rm|^2 + \frac{1}{C'}|\nabla \oo{\nabla} v|^4
\end{align*}
}
where we have used a commutation relation from Lemma \ref{commutation_relations} in the second step and $C'$ is to be determined. Using the previous result in combination with the assumption of a $C^0$-bound on $|\nabla v|$ gives
\begin{align*}
    \int_X |\nabla \nabla \oo{\nabla} v|^2 &\leq \int_X |\nabla \Delta v|^2 + C|\Rm|^2\\
    &+ \int_X \frac{24}{C'}|\nabla v|^2|\nabla \nabla \oo{\nabla} v|^2 + \frac{24}{C'}(\Delta v)^2 |\nabla \oo{\nabla} v|^2 + \frac{24}{C'} |\nabla v|^4 |\Ric|^2\\
    &\leq \int_X |\nabla \Delta v|^2 + C|\Rm|^2 + \frac{24C''}{C'} |\nabla \nabla \oo{\nabla} v|^2 + C (\Delta v)^2|\nabla \oo{\nabla} v|^2.
\end{align*}
We now choose $C'$ sufficiently large so that $\frac{24C''}{C'} \leq \frac{1}{2}$ and the result follows. Next,
{\allowdisplaybreaks
\begin{align*}
    \int_X |\nabla \nabla \nabla v|^2 &= -\int_X \nabla_{\oo{i}} \nabla_i \nabla_j \nabla_k v \cdot \nabla_{\oo{j}} \nabla_{\oo{k}} v\\
    &= -\int_X \nabla_{\oo{j}} \nabla_{\oo{k}} v (\nabla_j \nabla_k \Delta v + R_{j \oo{i}} \nabla_i \nabla_{k} v + R_{k\oo{i}}\nabla_j\nabla_{i} v\\
    &\quad\quad\quad\quad+ R_{j\oo{l}k\oo{i}} \nabla_i \nabla_l v +\nabla_j R_{k \oo{i}} \nabla_i v)\\
    &\leq \int_X \nabla_j \nabla_{\oo{j}} \nabla_{\oo{k}} v \cdot \nabla_k \Delta v + |\Ric||\nabla \nabla v|^2 + |\Ric||\nabla \nabla v|^2 + |\Rm||\nabla \nabla v|^2\\
    &+ \int_X |\nabla \Ric||\nabla v| |\nabla \nabla v|\\
    &\leq \int_X |\nabla \Delta v|^2 + |\Ric| |\nabla v| |\nabla \Delta v| + C|\Rm|^2 + \frac{1}{C'}|\nabla \nabla v|^4\\
    &+\int_X \frac{1}{2}|\nabla \Ric|^2 + \frac{1}{2}|\nabla v|^2 |\nabla \nabla v|^2\\
    &\leq \int_X \frac{3}{2}|\nabla \Delta v|^2 + \frac{1}{2} |\nabla v|^2 |\Ric|^2 + C|\Rm|^2 + \frac{1}{2}|\nabla\Ric|^2 + \frac{1}{2} |\nabla v|^2 |\nabla \nabla v|^2\\
    &+ \int_X \frac{1}{C'}|\nabla \nabla v|^4
\end{align*}
}
where  we have used a commutation relation from Lemma \ref{commutation_relations} in the second step and $C'$ is to be determined. Similarly, one finds that
\begin{align*}
    \int_X |\oo{\nabla} \nabla \nabla v|^2 &\leq \int_X \frac{3}{2}|\nabla \Delta v|^2 + \frac{1}{2} |\nabla v|^2 |\Ric|^2 + C|\Rm|^2 + \frac{1}{2}|\nabla\Ric|^2 + \frac{1}{2} |\nabla v|^2 |\nabla \nabla v|^2\\
    &+ \int_X \frac{1}{C'}|\nabla \nabla v|^4
\end{align*}
for the same $C'$. Combining these estimates and using the previous result again in combination with the $C^0$-bound on $|\nabla v|$ gives
\begin{align*}
    \int_X |\nabla \nabla \nabla v|^2 + |\oo{\nabla} \nabla\nabla v|^2 &\leq \int_X 3|\nabla \Delta v|^2 + C|\Rm|^2+ |\nabla \Ric|^2 + C|\nabla \nabla v|^2\\
    &+ \int_X \frac{32}{C'} |\nabla v|^2 (|\nabla \nabla \nabla v|^2 + |\oo{\nabla} \nabla \nabla v|^2)\\
    &+ \int_X \frac{32}{C'}(\Delta v)^2 |\nabla \nabla v|^2 + \frac{32}{C'} |\nabla v|^4|\Ric|^2\\
    &\leq \int_X 3|\nabla \Delta v|^2 + C|\Rm|^2+ |\nabla \Ric|^2 + C|\nabla \nabla v|^2\\
    &+ \int_X \frac{32C''}{C'} (|\nabla \nabla \nabla v|^2 + |\oo{\nabla} \nabla \nabla v|^2)+ C(\Delta v)^2 |\nabla \nabla v|^2.
\end{align*}
Now choose $C'$ sufficiently large so that $\frac{32C''}{C'} \leq \frac{1}{2}$ and the result is achieved.
\end{proof}

\noindent The $L^4$-like behaviour of the Ricci curvature immediately follows from the next result.

\begin{proposition}
There exists a constant $C > 0$ such that
\begin{equation}\label{mixed_hessian_bound}
    \lim_{\tau \to T} \int_{0}^{\tau} \int_X (T-t)^4(|\nabla \nabla u(t)|^2 + |\nabla \oo{\nabla} u(t)|^2)|\nabla \oo{\nabla} u(t)|^2 \omega(t)^n dt \leq C,
\end{equation}
and additionally,
\begin{align*}
    &\lim_{\tau \to T} \int_0^\tau \int_X (T-t)^4 |\nabla \nabla \oo{\nabla} u(t)|^2 \omega(t)^n dt \leq C.
\end{align*}
\end{proposition}

\begin{proof} From the two previous Propositions it is immediate that,
\begin{align}
    \int_X (|\nabla \nabla u|^2 + |\nabla \oo{\nabla} u|^2)|\nabla \oo{\nabla} u|^2 &\leq \frac{C}{(T-t)^4} + C \int_X |\nabla \nabla \oo{\nabla} u|^2 \nonumber\\
    &\leq \frac{C}{(T-t)^4} +C\int_X |\nabla \Delta u|^2 \label{intermediate_L4_estimate}.
\end{align}
Therefore, for any $\tau \in [0,T)$
\begin{align*}
    \int_{0}^\tau \int_X (T-t)^4(|\nabla \nabla u|^2 &+ |\nabla \oo{\nabla} u|^2)|\nabla \oo{\nabla} u|^2 \omega^n dt\\
    &\leq \int_0^\tau \left(C + C\int_X (T-t)^4 |\nabla \Delta u|^2 \omega^n \right) dt \leq C
\end{align*}
and the result follows.

\end{proof}

\noindent Alternatively, using (\ref{alternate_bound}) one finds there exists a $C>0$ such that for any $\tau \in [0,T)$ we have
\begin{equation}\label{alternate_hessian_bound}
    \int_0^{\tau} \int_X (|\nabla \nabla u|^2 + |\nabla \oo{\nabla} u|^2)|\nabla \oo{\nabla} u|^2 \omega^n dt \leq \frac{C}{(T-\tau)^3},
\end{equation}
\begin{equation*}
    \int_0^{\tau} \int_X |\nabla \nabla \oo{\nabla} u|^2 \omega^n dt \leq \frac{C}{(T-\tau)^3}.
\end{equation*}

\noindent We now prove the main result of Theorem \ref{main_theorem}. 

\begin{corollary}
There exists a constant $C > 0$ such that
\begin{equation}\label{Ricci_bound}
    \lim_{\tau \to T} \int_0^\tau \int_X (T-t)^8|\Ric(\omega(t))|^4 \omega(t)^n dt \leq C.
\end{equation}
\end{corollary}

\begin{proof}
Recall (\ref{Ricci_formula}). From this one can see
\begin{align*}
    (1-e^{t-T})^2|\Ric|^2 &= |e^{t-T}\omega-f^*\omega_Y -i\d\dbar u|^2\\
    &\leq (|e^{t-T} \omega| + |f^*\omega_Y| + |i\d\dbar u|)^2\\
    &\leq C(1 + |\nabla \oo{\nabla} u|)^2\\
    &\leq C (1 + |\nabla \oo{\nabla} u|^2)
\end{align*}
and so
\begin{equation*}
    |\Ric|^2 \leq \frac{C}{(T-t)^2} (1+|\nabla \oo{\nabla} u|^2).
\end{equation*}
Thus for any $\tau \in [0,T)$ we have
\begin{align*}
    \int_0^\tau \int_X (T-t)^8 |\Ric|^4 \leq \int_0^\tau \int_X C(T-t)^4(1+|\nabla\oo{\nabla} u|^4) \leq C.
\end{align*}
The result follows.
\end{proof}

\noindent Alternatively, if one uses (\ref{alternate_hessian_bound}) we find there exists a $C >0$ such that for any $\tau \in [0,T)$
\begin{equation}\label{alternate_Ricci_bound}
    \int_0^\tau \int_X (T-t)^4 |\Ric|^4\omega^n dt \leq \frac{C}{(T-\tau)^3}.
\end{equation}

\begin{remark}
If we assume $n=2$, that is, $X$ has real dimension $4$, we can apply Simon's result, \cite[Thm 1.1]{S15}, in combination with the uniform scalar curvature estimate to achieve (\ref{alternate_Ricci_bound}).
\end{remark}

\subsection{$L^4$-Behaviour of the $(2,0)$-Hessian of the Ricci Potential}

\noindent Let us also derive an $L^4$-type estimate similar to (\ref{mixed_hessian_bound}) for the remaining complex Hessian $\nabla \nabla u$. This requires the $L^2$-behaviour of $\nabla \Ric$ in Proposition \ref{triple_nabla_bounds}.

\begin{proposition}
There exists a constant $C > 0$ such that 
\begin{equation*}
    \lim_{\tau \to T}\int_0^\tau \int_X (T-t)^8|\nabla \Ric|^2 \omega(t)^n dt \leq C.
\end{equation*}
\end{proposition}

\begin{proof}
From the evolution of the norm of the Ricci curvature, (\ref{Ricci_norm_evolution}), we have
\begin{align*}
    \int_X |\nabla \Ric|^2 &= \int_X |\Ric|^2 + \Real\ip{\Rm, \oo{\Ric \otimes \Ric}} - \frac{1}{2}  \d_t |\Ric|^2\\
    &\leq \int_X \left(|\Ric|^2 + |\Rm||\Ric|^2 +\frac{1}{2}(-n-R)|\Ric|^2\right) -\frac{1}{2}\frac{d}{dt}\int_X |\Ric|^2
\end{align*}
and therefore
\begin{align*}
    \int_X |\nabla \Ric|^2 \leq \int_X \left(C|\Rm|^2 + |\Ric|^4\right) -\frac{1}{2} \frac{d}{dt}\int_X |\Ric|^2.
\end{align*}
It follows that
\begin{align*}
    \int_0^\tau\int_X (T-t)^8 |\nabla \Ric|^2 &\leq \int_0^\tau\int_X  C(T-t)^8|\Rm|^2 + (T-t)^8 |\Ric|^4\\
    &+C- \frac{1}{2}\int_0^\tau \frac{d}{dt} \int_X (T-t)^8 |\Ric|^2\\
    &\leq C
\end{align*}
where $C$ is independent of $\tau$.
\end{proof}

\noindent This completes the proof of Theorem \ref{main_theorem}.
Alternatively, using (\ref{alternate_Ricci_bound}) we may find a constant $C > 0$ such that for all $\tau \in [0,T)$
\begin{equation}\label{alternate_gradient_Ricci_bound}
    \int_0^\tau \int_X (T-t)^4|\nabla \Ric|^2 \omega^n dt \leq \frac{C}{(T-\tau)^3}.
\end{equation}

\begin{proposition}
There exists a constant $C > 0$ such that
\begin{equation*}
    \lim_{\tau \to T} \int_0^\tau \int_X (T-t)^8|\nabla \nabla u(t)|^4 \omega(t)^n dt \leq C
\end{equation*}
and additionally,
\begin{equation*}
    \lim_{\tau \to T} \int_0^\tau \int_X(T-t)^8 (|\nabla \nabla \nabla u(t)|^2 + |\oo{\nabla} \nabla \nabla u(t)|^2) \omega(t)^n dt \leq C.
\end{equation*}
\end{proposition}

\begin{proof}
We have
\begin{align*}
    \int_X (|\nabla \nabla u|^2 + |\nabla \oo{\nabla} u|^2)|\nabla \nabla u|^2 &\leq \frac{C}{(T-t)^4} + C \int_X |\nabla \nabla \nabla u|^2 + |\oo{\nabla} \nabla \nabla u|^2\\
    &\leq \frac{C}{(T-t)^4} + C\int_X |\nabla \Delta u|^2 + |\nabla \Ric|^2.
\end{align*}
Therefore, 
\begin{align*}
    \int_0^\tau \int_X (T-t)^8 |\nabla \nabla u|^4 \omega^n dt &\leq C+ C \int_0^\tau \int_X (T^4(T-t)^4|\nabla \Delta u|^2 + (T-t)^8|\nabla \Ric|^2)\omega^n dt\\
    &\leq C
\end{align*}
and the result follows.
\end{proof}

\noindent If instead (\ref{alternate_gradient_Ricci_bound}) is used, then one finds there exists a $C>0$ such that for any $\tau \in [0,T)$ we have
\begin{equation*}
    \int_0^\tau \int_X (T-t)^4|\nabla \nabla u|^4 \omega^ndt \leq \frac{C}{(T-\tau)^3},
\end{equation*}
\begin{equation*}
    \int_0^\tau \int_X (T-t)^4(|\nabla \nabla \nabla u|^2 + |\oo{\nabla} \nabla \nabla u|^2) \omega^n dt \leq \frac{C}{(T-\tau)^3}.
\end{equation*}

\subsection{Interpolation Estimates}

\noindent Recall equation (\ref{Rm_bound}) from which we have
\begin{equation}\label{Good_L2_Ricci_bound}
    \int_X |\Ric|^2 \omega^n \leq \frac{C}{(T-t)^2}.
\end{equation}
We combine this with the $L^4$-like estimate on the Ricci curvature, (\ref{Ricci_bound}), to deduce the following interpolating estimate.

\begin{corollary}
For any $2 \leq p \leq 4$ there exists a constant $C=C(p,\omega_0, f^*\omega_Y, T) > 0$ such that
\begin{equation*}
    \lim_{\tau \to T} \int_0^\tau \int_X (T-t)^{3p-4}|\Ric(\omega(t))|^p \omega(t)^n dt \leq C.
\end{equation*}
\end{corollary}

\begin{proof}
This follows by using (\ref{Ricci_bound}) and (\ref{Good_L2_Ricci_bound}) in combination with the H\"older inequality $\norm{fg}_{1} \leq \norm{f}_{r} \norm{g}_{s}$ where
\begin{equation*}
    f = (T-t)^{4-p}|\Ric|^{4-p}, \quad g= (T-t)^{4p-8} |\Ric|^{2p-4}, \quad r = \frac{2}{4-p}, \quad s = \frac{2}{p-2}.
\end{equation*}
That is, for all $\tau \in [0,T)$ we have
\begin{align*}
    \int_0^\tau \int_X (T-t)^{3p-4}|\Ric|^p &\leq \left(\int_0^\tau \int_X (T-t)^2|\Ric|^2\right)^{\frac{4-p}{2}} \left(\int_0^\tau \int_X (T-t)^8 |\Ric|^4\right)^{\frac{p-2}{2}}\\
    &\leq C.
\end{align*}
\end{proof}

\noindent We also have the following interpolating estimate on the behaviour of the scalar curvature.

\begin{corollary}
For any $1 \leq p$ there exists a $C=C(p, \omega_0, f^*\omega_Y, T) > 0$ such that for all $t \in [0, T)$ we have
\begin{equation*}
    \int_X |R|^p \omega^n \leq \frac{C}{(T-t)^{2p-2}}.
\end{equation*}
\end{corollary}

\begin{proof}
Let $P > p$. We use equation (\ref{L1_scalar_bound}) in combination with the H\"older inequality $\norm{fg}_{L^1} \leq \norm{f}_{L^r} \norm{g}_{L^s}$ where
\begin{equation*}
    f = |R|^{\frac{P-p}{P-1}}, \quad g = (T-t)^{\frac{2P(p-1)}{P-1}} |R|^{\frac{P(p-1)}{P-1}}, \quad r = \frac{P-1}{P-p}, \quad s = \frac{P-1}{p-1}.
\end{equation*}
That is,
\begin{align*}
    \int_X (T-t)^{\frac{2P(p-1)}{P-1}} |R|^p &\leq \left(\int_X |R|\right)^{\frac{P-p}{P-1}} \left(\int_X (T-t)^{2P}|R|^{P}\right)^{\frac{p-1}{P-1}}\\
    &\leq C^{\frac{P-p}{P-1}} \left(\left(\int_X (T-t)^{2P}|R|^{P}\right)^{\frac{1}{P}}\right)^{\frac{P(p-1)}{P-1}}.
\end{align*}
Now taking the limit as $P \to \infty$ and applying Theorem \ref{bounds} gives us
\begin{equation*}
    \int_X (T-t)^{2(p-1)}|R|^p \leq C \left(\norm{(T-t)^2 R}_{C^0}\right)^{p-1} \leq C.
\end{equation*}
\end{proof}

\section{On Improving the Global $L^p$-Estimates}

In this section we provide further estimates which are dependent on an $L^2$-like norm of $\Delta \tr_\omega (f^*\omega_Y)$. This is motivated by the following result which would enable us to upgrade the $L^4$-like Ricci curvature estimate to a genuine $L^4$-estimate provided that the $L^2$-behaviour of $\Delta \tr_\omega(f^*\omega_Y)$ is known. In fact, this would enable most estimates in the previous section with an additional time integral to be improved.

\begin{proposition}
There exists a $C > 0$ such that for all $\tau \in [0,T)$ we have
\begin{align*}
    \int_X |\nabla \Delta u (\tau)|^2 \omega(\tau)^n \leq \frac{C}{(T-\tau)^3} + C \int_0^\tau \int_X (\Delta \tr_{\omega(t)} (f^*\omega_Y))^2 \omega(t)^n dt,
\end{align*}
and additionally
\begin{align*}
    \int_0^\tau \int_X (|\nabla \nabla \Delta u(t)|^2 &+ |\nabla \oo{\nabla} \Delta u(t)|^2) \omega(t)^n dt\\
    &\leq \frac{C}{(T-\tau)^3} + C \int_0^\tau \int_X (\Delta \tr_{\omega(t)} (f^*\omega_Y))^2 \omega(t)^n dt.
\end{align*}
\end{proposition}

\begin{proof} This proof is similar to \cite[Thm 4.5]{TZ16}. First we combine the evolution equation for $\Delta u$ with equation (\ref{Ricci_formula}),
\begin{equation*}
    \parabolic{}\Delta u = (1+e^{t-T}) \Delta u - \ip{f^*\omega_Y, i\d\dbar u} - |\nabla \oo{\nabla} u|^2 + \Delta \tr_\omega(f^*\omega_Y)
\end{equation*}
and calculate,
\begin{align*}
    \parabolic{} |\nabla \Delta u|^2 &= |\nabla \Delta u|^2 - |\nabla \nabla \Delta u|^2 - |\nabla \oo{\nabla} \Delta u|^2\\
    &+ 2 \Real \ip{\nabla \parabolic{} \Delta u, \oo{\nabla} \Delta u}.
\end{align*}
Therefore, we have
{\allowdisplaybreaks
\begin{align*}
    \frac{d}{dt} \int_X |\nabla \Delta u|^2 \omega^n &= \int_X \frac{\d}{\d t}\left( |\nabla \Delta u|^2\right) \omega^n + \int_X |\nabla \Delta u|^2 (-n-R) \omega^n\\
    &= \int_X |\nabla \Delta u|^2(1-n-R)  -|\nabla \nabla \Delta u|^2 - |\nabla \oo{\nabla} \Delta u|^2\\
    &+ \int_X 2 (1+e^{t-T})|\nabla \Delta u|^2-2\Real \ip{\nabla \ip{f^*\omega_Y, i\d\dbar u}, \oo{\nabla} \Delta u}\\
    &+ \int_X - 2 \Real \ip{\nabla |\nabla \oo{\nabla} u|^2, \oo{\nabla} \Delta u} + 2\Real \ip{\nabla \Delta \tr_\omega(f^*\omega_Y), \oo{\nabla} \Delta u}\\
    &\leq \int_X C|\nabla \Delta u|^2 - |\nabla \nabla \Delta u|^2 - |\nabla \oo{\nabla} \Delta u|^2\\
    &+\int_X -2\Real \ip{\nabla \ip{f^*\omega_Y, i\d\dbar u}, \oo{\nabla} \Delta u} - 2 \Real \ip{\nabla |\nabla \oo{\nabla} u|^2, \oo{\nabla} \Delta u}\\
    &+\int_X 2\Real \ip{\nabla \Delta \tr_\omega(f^*\omega_Y), \oo{\nabla} \Delta u}.
\end{align*}
}
Let us estimate the last three terms independently. Note that,
\begin{align*}
    \int_X -2 \Real \ip{\nabla \ip{f^*\omega_Y, i\d\dbar u}, \oo{\nabla} \Delta u}
    &= \int_X 2\ip{f^*\omega_Y, i\d\dbar u} \Delta \Delta u\\
    &\leq \int_X 2|f^*\omega_Y| |\nabla \oo{\nabla} u| |\Delta \Delta u|\\
    &\leq \int_X C |\nabla \oo{\nabla} u|^2 + \frac{1}{4} |\nabla \oo{\nabla} \Delta u|^2.
\end{align*}
Next, using integration by parts and applying (\ref{intermediate_L4_estimate}) yields
\begin{align*}
    \int_X -2 \Real \ip{\nabla |\nabla \oo{\nabla} u|^2, \oo{\nabla} \Delta u} &= \int_X 2 |\nabla \oo{\nabla} u|^2 \Delta \Delta u\\
    &\leq \int_X C|\nabla \oo{\nabla} u|^4 + \frac{1}{4} |\nabla \oo{\nabla} \Delta u|^2\\
    &\leq \frac{C}{(T-t)^4} + \int_X C |\nabla \Delta u|^2 + \frac{1}{4} |\nabla \oo{\nabla} \Delta u|^2.
\end{align*}
For the final term, we have
\begin{align*}
    \int_X 2\Real \ip{\nabla \Delta \tr_\omega(f^*\omega_Y), \oo{\nabla} \Delta u} &=- \int_X 2 \Delta \tr_\omega(f^*\omega_Y) \Delta \Delta u\\
    &\leq \int_X C(\Delta \tr_\omega(f^*\omega_Y))^2 + \frac{1}{4} |\nabla \oo{\nabla} \Delta u|^2.
\end{align*}
Combining these estimates yields
\begin{align*}
    \frac{d}{dt} \int_X |\nabla \Delta u|^2 &\leq \frac{C}{(T-t)^4} + \int_X C|\nabla \Delta u|^2 - \frac{1}{4} |\nabla \nabla \Delta u|^2 - \frac{1}{4} |\nabla \oo{\nabla} \Delta u|^2 + C(\Delta \tr_\omega(f^*\omega_Y))^2.
\end{align*}
It follows, by the Fundamental Theorem of Calculus, that for any $\tau \in [0,T)$ we have
\begin{align*}
    \int_X |\nabla \Delta u(\tau)|^2 \omega(\tau)^n - \int_X |\nabla \Delta u(0)|^2 \omega_0^n &\leq \frac{C}{(T-\tau)^3} - C\\
    &+ C\int_0^\tau \int_X (|\nabla \Delta u(t)|^2 + (\Delta \tr_{\omega(t)}(f^*\omega_Y))^2) \omega(t)^n dt\\
    &-\frac{1}{4} \int_0^\tau \int_X (|\nabla \nabla \Delta u(t)|^2 + |\nabla \oo{\nabla} \Delta u(t)|^2) \omega(t)^n dt.
\end{align*}
By applying (\ref{alternate_bound}), we see
\begin{align*}
    \int_X |\nabla \Delta u(\tau)|^2\omega(\tau)^n &+ \frac{1}{4} \int_0^\tau \int_X (|\nabla \nabla \Delta u(t)|^2 + |\nabla \oo{\nabla} \Delta u(t)|^2) \omega(t)^n dt\\ 
    &\leq \frac{C}{(T-\tau)^3} + C\int_0^\tau \int_X (\Delta \tr_{\omega(t)}(f^*\omega_Y))^2 \omega(t)^n dt.
\end{align*}
\end{proof}

\noindent If this result is applied to the Ricci curvature,  i.e. (\ref{intermediate_L4_estimate}) and the proof of (\ref{Ricci_bound}), we see that there exists a $C > 0$ such that for all $\tau \in [0,T)$ we have
\begin{equation}\label{actual_L4_Ricci_bound}
    \int_X |\Ric(\omega(\tau))|^4 \omega(\tau)^n \leq \frac{C}{(T-\tau)^8} + \frac{C}{(T-\tau)^4} \int_0^\tau \int_X (\Delta \tr_{\omega(t)}(f^*\omega_Y))^2 \omega(t)^n dt. 
\end{equation}
Also note that
\begin{align*}
    \int_0^\tau \int_X (\Delta \Delta u(t))^2 \omega(t)^n dt &= \int_0^\tau \int_X |\nabla \oo{\nabla} \Delta u(t)|^2 \omega(t)^n dt\\
    &\leq \frac{C}{(T-\tau)^3} + C\int_0^\tau \int_X (\Delta \tr_{\omega(t)}(f^*\omega_Y))^2 \omega(t)^n dt.  
\end{align*}

\begin{proposition}
There exists a constant $C > 0$ such that for all $t \in [0,T)$ we have
\begin{equation*}
    \int_X |\nabla \nabla \tr_\omega(f^*\omega_Y)|^2 \omega^n \leq \frac{C}{(T-t)^2} + \int_X (\Delta \tr_\omega (f^*\omega_Y))^2 \omega^n.
\end{equation*}
\end{proposition}

\begin{proof}
In a similar vein as the proof of Proposition \ref{L2_unmixed_hessian_bound} one has
\begin{align*}
    \int_X |\nabla \nabla \tr_\omega(f^*\omega_Y)|^2 \leq \int_X (\Delta \tr_\omega(f^*\omega_Y))^2 + |\Ric| |\nabla \tr_\omega(f^*\omega_Y)|^2.
\end{align*}
Then using inequality (\ref{lower_trace_bound}), we have
\begin{align*}
    \int_X |\nabla \nabla \tr_\omega(f^*\omega_Y)|^2 &\leq \int_X (\Delta \tr_\omega(f^*\omega_Y))^2 + C|\Ric|(\Delta \tr_\omega(f^*\omega_Y) + |\Ric|+1)\\
    &\leq C\int_X (\Delta \tr_\omega(f^*\omega_Y))^2 + |\Ric|^2 + |\Ric|\\
    &\leq \frac{C}{(T-t)^2} + C\int_X (\Delta \tr_\omega (f^*\omega_Y))^2.
\end{align*}
\end{proof}

\noindent Finally, we can apply the previous two Propositions to derive estimates for the complex Hessians $\nabla \oo{\nabla} R$ and $\nabla \nabla R$ of the scalar curvature.

\begin{corollary}
There exists a $C > 0$ such that for all $\tau \in [0,T)$ we have
\begin{align*}
    \int_0^\tau \int_X (T-t)^2(|\nabla \oo{\nabla} R(\omega(t))|^2 &+ |\nabla \nabla R(\omega(t))|^2) \omega(t)^n dt\\
    &\leq \frac{C}{(T-\tau)^3} + C \int_0^\tau \int_X (\Delta \tr_{\omega(t)}(f^*\omega_Y))^2 \omega(t)^n dt.
\end{align*}
\end{corollary}

\begin{proof}
Recall equation (\ref{scalar_formula}). From this it follows that
\begin{align*}
    (1-e^{t-T})^2 |\nabla \oo{\nabla} R|^2 &= |\nabla \oo{\nabla} \tr_\omega(f^*\omega_Y) + \nabla \oo{\nabla} \Delta u|^2\\
    &\leq (|\nabla \oo{\nabla} \tr_\omega(f^*\omega_Y)| + |\nabla \oo{\nabla} \Delta u|)^2\\
    &\leq 2|\nabla \oo{\nabla} \tr_\omega(f^*\omega_Y)|^2 + 2|\nabla \oo{\nabla} \Delta u|^2
\end{align*}
and likewise
\begin{equation*}
    (1-e^{t-T})^2|\nabla \nabla R|^2 \leq 2|\nabla \nabla \tr_\omega(f^*\omega_Y)|^2 + 2 |\nabla \nabla \Delta u|^2.
\end{equation*}
Therefore, for any $\tau \in [0,T)$ we have
\begin{align*}
    \int_0^\tau \int_X (T-t)^2 |\nabla \oo{\nabla} R(\omega(t))|^2\omega(t)^n dt &\leq C\int_0^\tau \int_X (|\nabla \oo{\nabla} \Delta u(t)|^2 + |\nabla \oo{\nabla} \tr_{\omega(t)}(f^*\omega_Y)|^2) \omega(t)^n dt\\
    &\leq \frac{C}{(T-\tau)^3} + C\int_0^\tau \int_X (\Delta \tr_{\omega(t)}(f^*\omega_Y))^2 \omega(t)^n dt,
\end{align*}
and likewise for $\nabla \nabla R$.
\end{proof}

\noindent Also note that
\begin{equation*}
    \int_X (\Delta R)^2 \omega^n = \int_X |\nabla \oo{\nabla} R|^2 \omega^n,
\end{equation*}
so that one can state a similar result for the Laplacian of the scalar curvature. 

\section{Conversion Between K\"ahler-Ricci Flows}

\noindent We conclude with the conversion of the $L^4$-like estimate of the Ricci curvature to one for the unnormalized flow. One can perform similar calculations for other estimates in Theorem \ref{main_theorem}.\\

\noindent The time and metric rescalings $s = e^t-1$ and $\hat{\omega}(s) = e^t \omega(t)$, lets one see that the normalized K\"ahler-Ricci flow (\ref{KRF}) is equivalent to the following unnormalized K\"ahler-Ricci flow
\begin{equation}\label{UKRF}\tag{UKRF}
    \frac{\d \hat{\omega}(s)}{\d s} = -\Ric(\hat{\omega}(s)), \quad \hat{\omega}(0) = \omega_0,
\end{equation}
where $s \in [0,S)$ for $S = e^T-1 < \infty$.\\
The $L^4$-like estimate on the Ricci-curvature, (\ref{Ricci_bound}), can be converted to an estimate for the unnormalized K\"ahler-Ricci flow as follows. For any $\tau \in [0,T)$ let $\sigma = e^\tau-1 \in [0,S)$, then
\begin{align*}
    C &\geq \int_0^\tau \int_X (T-t)^8 |\Ric (\omega(t))|^4_{\omega(t)} \omega(t)^n dt\\
    &\geq \int_0^\tau \int_X (1-e^{t-T})^8 |\Ric (\omega(t))|^4_{\omega(t)} \omega(t)^n dt\\
    &= \int_0^\sigma \int_X (1-e^{t-T})^8 |\Ric (\omega(t))|^4_{\omega(t)} \frac{1}{(1+s)^{n+1}} \hat{\omega}(s)^n ds\\
    &= \int_0^\sigma \int_X \frac{e^{-8T}(S-s)^8(1+s)^4}{(1+s)^{n+1}} |\Ric (\hat{\omega}(s))|^4_{\hat{\omega}(s)} \hat{\omega}(s)^n ds\\
    &\geq e^{-8T}\frac{1}{(1+S)^{n+1}} \int_0^\sigma \int_X (S-s)^8 |\Ric (\hat{\omega}(s))|_{\hat{\omega}(s)}^4 \hat{\omega}(s)^n ds.
\end{align*}
Therefore, there exists a $C > 0$ such that
\begin{equation*}
    \lim_{\sigma \to S} \int_0^\sigma \int_X (S-s)^8 |\Ric (\hat{\omega}(s))|_{\hat{\omega}(s)}^4 \hat{\omega}(s)^n ds \leq C.
\end{equation*}
That is, recall the existence of a holomorphic map $f: X \to Y$ and a K\"ahler metric $\omega_Y$ satisfying $[f^*\omega_Y] = [\omega(T)]$. This second assumption in terms of the unnormalized flow is still $[f^*\hat{\omega}_Y] = [\hat{\omega}(S)]$ where $\hat{\omega}_Y = e^T \omega_Y$ is still some K\"ahler metric on a neighbourhood of $Y$.

\section{Appendix}
\subsection{Commuting Covariant Derivatives}

\noindent Recall the notion of normal coordinates.

\begin{proposition}[Normal Coordinates]
Let $X$ be a K\"ahler manifold with K\"ahler metric $g$ and suppose $T$ a Hermitian tensor, that is, $\oo{T_{i\oo{j}}} = T_{\oo{j}i}$. Then for any $x \in X$ there exists local holomorphic coordinates, called \textbf{normal coordinates}, such that
\begin{equation*}
    g_{i\oo{j}}(x) = \delta_{i\oo{j}}, \quad \d_k g_{i\oo{j}}(x) = 0, \quad \Gamma^k_{ij}(x) = 0, \quad T_{i\oo{j}}(x) = \lambda_i \delta_{i\oo{j}},
\end{equation*}
for some $\lambda_i \in \R$.
\end{proposition}

\noindent Additionally, recall that the components of the covariant derivative of a tensor can be expressed in local coordinates in the following manner.

\begin{proposition}
Let $\nabla$ and $\oo{\nabla}$ denote the $(1,0)$ and $(0,1)$ components of the Chern connection on the K\"ahler manifold $(X,\omega)$. Suppose
\begin{equation*}
    T = T_{i_1 \dots i_p \oo{j}_1 \dots \oo{j}_q} dz^{i_1} \otimes \dots \otimes dz^{i_p} \otimes d \oo{z}^{j_1} \otimes \dots \otimes d\oo{z}^{j_q}
\end{equation*}
and denote
\begin{align*}
    \nabla_k T_{i_1 \dots i_p \oo{j}_1 \dots \oo{j}_q} &= (\nabla T)(\d_k, \d_{i_1}, \dots, \d_{i_p}, \dbar_{j_1}, \dots, \dbar_{j_q}),\\
    \nabla_{\oo{l}} T_{i_1 \dots i_p \oo{j}_1 \dots \oo{j}_q} &= (\oo{\nabla} T)(\dbar_l, \d_{i_1}, \dots, \d_{i_p}, \dbar_{j_1}, \dots, \dbar_{j_q}).
\end{align*}
Then in any local holomorphic coordinates we have
\begin{align*}
    \nabla_k T_{i_1 \dots i_p \oo{j}_1 \dots \oo{j}_q} &= \d_k T_{i_1 \dots i_p \oo{j}_1 \dots \oo{j}_q} - \sum_{\alpha=1}^p \Gamma^{r}_{ki_\alpha} T_{i_1 \dots r \dots i_p \oo{j}_1 \dots \oo{j}_q},\\
    \nabla_{\oo{l}} T_{i_1 \dots i_p \oo{j}_1 \dots \oo{j}_q} &= \dbar_l T_{i_1 \dots i_p \oo{j}_1 \dots \oo{j}_q} - \sum_{\beta=1}^q \Gamma^{\oo{s}}_{\oo{l} \oo{j}_\beta} T_{i_1 \dots i_p \oo{j}_1 \dots \oo{s} \dots \oo{j}_q}.
\end{align*}
\end{proposition}

\noindent We use these two tools in computing the following commutation relations.

\begin{lemma}\label{commutation_relations}
We have the following commutation relations under normal coordinates
\begin{align*}
    \nabla_i \nabla_{\oo{i}} \nabla_{\oo{j}} \nabla_k v &= \nabla_{\oo{j}} \nabla_k \Delta v + R_{i \oo{j}} \nabla_{\oo{i}} \nabla_k v + R_{k\oo{j}l\oo{i}} \nabla_i \nabla_{\oo{l}} v,\\
    \nabla_{\oo{i}} \nabla_i \nabla_j \nabla_k v &= \nabla_j \nabla_k \Delta v + R_{j \oo{i}} \nabla_i \nabla_{k} v + R_{k\oo{i}}\nabla_j\nabla_{i} v + R_{j\oo{l}k\oo{i}} \nabla_i \nabla_l v +\nabla_j R_{k \oo{i}} \nabla_i v.
\end{align*}
\end{lemma}

\begin{proof}
For the second commutation relation we have the following calculation under normal coordinates
\begin{align*}
    \nabla_{\oo{i}} \nabla_i \nabla_j \nabla_k v &= \dbar_i \nabla_i \nabla_j \nabla_k v\\
    &= \dbar_i (\d_i \nabla_j \nabla_k v - \Gamma^l_{ij} \nabla_l \nabla_k v-\Gamma^l_{ik} \nabla_j \nabla_l v)\\
    &= \dbar_i \d_i (\d_j \nabla_k v - \Gamma^{l}_{jk} \nabla_l v) - (\dbar_i \Gamma^l_{ij}) \nabla_l \nabla_k v - (\dbar_i \Gamma^l_{ik}) \nabla_j \nabla_l v\\
    &= \dbar_i \d_i \d_j \d_k v - \dbar_i \d_i (\Gamma^{l}_{jk} \nabla_l v) + R_{j\oo{l}} \nabla_l \nabla_k v + R_{k\oo{l}} \nabla_j \nabla_l v
\end{align*}
where
\begin{align*}
    -\dbar_i \d_i (\Gamma^l_{jk} \nabla_l v) &= -(\dbar_i \d_i \Gamma^l_{jk}) \nabla_l v -(\dbar_i \Gamma^l_{jk}) \d_i \nabla_l v - (\d_i \Gamma^l_{jk}) \dbar_i \nabla_l v\\
    &= \d_i R_{j}{}^{l}{}_{k\oo{i}} \nabla_l v + R_{j\oo{l}k\oo{i}} \nabla_i \nabla_l v - (\d_i \Gamma^l_{jk}) \nabla_{\oo{i}} \nabla_l v\\
    &= \nabla_k R_{j\oo{l}} \nabla_l v + R_{j\oo{l}k\oo{i}} \nabla_i \nabla_l v - (\d_i \Gamma^l_{jk}) \nabla_{\oo{i}} \nabla_l v.
\end{align*}
Also,
\begin{align*}
    \nabla_j \nabla_k \Delta v &= \d_j \d_k (g^{i\oo{l}} \d_i \dbar_l v)\\
    &= (\d_j \d_k g^{i\oo{l}}) \d_i \dbar_l v + g^{i\oo{l}}\d_j \d_k \d_i \dbar_l v\\
    &= -g^{i\oo{q}} g^{p\oo{l}} (\d_j \d_k g_{p\oo{q}}) \nabla_i \nabla_{\oo{l}} v + \d_j \d_k \d_i \dbar_i v\\
    &= -g^{p\oo{l}}(\d_j \Gamma^{i}_{kp}) \nabla_i \nabla_{\oo{l}} v + \d_j \d_k \d_i \dbar_i v\\
    &=-(\d_j \Gamma^{i}_{kl}) \nabla_i \nabla_{\oo{l}} v + \d_j \d_k \d_i \dbar_i v.
\end{align*}
Therefore,
\begin{align*}
    \nabla_{\oo{i}} \nabla_i \nabla_j \nabla_k v &= \nabla_j \nabla_k \Delta v + (\d_j \Gamma^i_{kl}) \nabla_i \nabla_{\oo{l}} v+\nabla_k R_{j\oo{l}} \nabla_l v + R_{j\oo{l}k\oo{i}} \nabla_i \nabla_l v\\
    &-(\d_i \Gamma^l_{jk})\nabla_{\oo{i}} \nabla_l v +R_{j\oo{l}} \nabla_l \nabla_k v + R_{k\oo{l}} \nabla_j \nabla_l v
\end{align*}
and the result follows since $\d_q \Gamma^k_{ij} = \d_q (g^{k\oo{l}} \d_i g_{j\oo{l}}) = g^{k\oo{l}} \d_q\d_i g_{j\oo{l}} = \d_i \Gamma^k_{qj}$ in normal coordinates.
\end{proof}

\subsection{Additional Equations}

\noindent We provide a proof of Bochner's formula.

\begin{proposition}[Bochner's Formula]
Let $v: X \to \R$ be an arbitrary, smooth, real function. Then
\begin{equation}\label{Bochner_formula}
    \Delta |\nabla v|^2 = |\nabla \nabla v|^2 + |\nabla \oo{\nabla} v|^2 + \ip{\nabla \Delta v, \oo{\nabla} v} + \ip{\nabla v, \oo{\nabla} \Delta v} + \ip{\Ric, \oo{\nabla} v \otimes \nabla v}.
\end{equation}
\end{proposition}

{\allowdisplaybreaks
\begin{proof}
Using normal coordinates,
\begin{align*}
    \Delta |\nabla v|^2 &= g^{i\oo{j}} \d_i \dbar_j (g^{k\oo{l}} \d_k v \dbar_l v)\\
    &= g^{i\oo{j}} (\d_i \dbar_j g^{k\oo{l}}) \d_k v \dbar_l v + g^{i\oo{j}} g^{k\oo{l}}  \d_i \dbar_j (\d_k v \dbar_l v)\\
    &= -g^{i\oo{j}} g^{p\oo{l}} g^{k\oo{q}} (\d_i \dbar_j g_{p\oo{q}}) \d_k v \dbar_l v\\
    &+ g^{i\oo{j}} g^{k\oo{l}} [(\d_i \dbar_j\d_k v) \dbar_l v + (\dbar_j\d_k v) (\d_i \dbar_l v)\\
    &+ (\d_i \d_k v) (\dbar_j \dbar_l v) + \d_k v (\d_i \dbar_j\dbar_l v)]\\
    &= g^{i\oo{j}} g^{p\oo{l}} g^{k\oo{q}} R_{i\oo{j}p\oo{q}} \d_k v \dbar_l v + g^{k\oo{l}} \d_k (g^{i\oo{j}}\d_i \dbar_j v) \dbar_l v + g^{i\oo{j}} g^{k\oo{l}} (\dbar_j\d_k v) (\d_i \dbar_l v)\\
    &+ g^{i\oo{j}} g^{k\oo{l}} (\d_i \d_k v) (\dbar_j \dbar_l v) + g^{k\oo{l}} (\d_k v) \dbar_l (g^{i\oo{j}}\d_i \dbar_j v)\\
    &= g^{p\oo{l}} g^{k\oo{q}} R_{p\oo{q}} \nabla_k v \nabla_{\oo{l}} v + g^{k\oo{l}} \nabla_k (\Delta v) \nabla_{\oo{l}} v + g^{i\oo{j}} g^{k\oo{l}} (\nabla_i \nabla_{\oo{l}} v)(\nabla_{\oo{j}}\nabla_k v) \\
    &+ g^{i\oo{j}} g^{k\oo{l}} (\nabla_i \nabla_k v) (\nabla_{\oo{j}} \nabla_{\oo{l}} v) + g^{k\oo{l}} (\nabla_k v) \nabla_{\oo{l}} (\Delta v)\\ 
    &= \ip{\Ric, \oo{\nabla} v \otimes \nabla v} +\ip{\nabla \Delta v, \oo{\nabla} v} + |\nabla \oo{\nabla} v|^2
    + |\nabla \nabla v|^2 + \ip{\nabla v, \oo{\nabla} \Delta v}.
\end{align*}
\end{proof}
}

\noindent In a similar manner one can calculate the following evolution equation.

\begin{proposition}
The norm of the Ricci curvature evolves according to
\begin{equation}\label{Ricci_norm_evolution}
    \parabolic{} |\Ric|^2 = 2|\Ric|^2-2|\nabla \Ric|^2 +2\Real \ip{\Rm, \oo{\Ric \otimes \Ric}}.
\end{equation}
\end{proposition}

\begin{proof}
Using normal coordinates, one calculates
\begin{align*}
    \d_t |\Ric|^2 = 2|\Ric|^2 + 2\Real \ip{\Ric, \oo{\nabla} \nabla R} + 2R_{i\oo{j}}R_{j\oo{k}}R_{k\oo{i}},
\end{align*}
and
\begin{align*}
    \Delta |\Ric|^2 = 2|\nabla \Ric|^2 + 2\Real \ip{\Ric, \oo{\nabla} \nabla R} - 2\Real \ip{\Rm, \oo{\Ric \otimes \Ric}} + 2R_{i\oo{j}} R_{j\oo{k}} R_{k\oo{i}}.
\end{align*}
\end{proof}

\end{document}